\documentclass[10pt]{amsart}
\allowdisplaybreaks

\usepackage{amsfonts}
\usepackage{amsmath}
\usepackage{amssymb}
\usepackage{amsthm}
\usepackage{tabularx}
\usepackage{xcolor}
\usepackage{soul}
\usepackage{graphicx} \usepackage{enumerate} \usepackage{multicol}
\usepackage{mathrsfs} \usepackage[all,cmtip]{xy}
\usepackage[color=blue!20]{todonotes}

\newcounter{dummy} \numberwithin{dummy}{section}
\newtheorem{theorem}[dummy]{Theorem}
\newtheorem{corollary}[dummy]{Corollary}

\newtheorem{lemma}[dummy]{Lemma}
\newtheorem{definition}[dummy]{Definition}
\newtheorem{proposition}[dummy]{Proposition}
\theoremstyle{remark}
\newtheorem{remark}[dummy]{Remark}
\newtheorem{example}[dummy]{Example}

\newcommand{\calA}{\mathcal{A}}

\newcommand{\scrE}{\mathscr{E}}

\DeclareMathOperator{\Ann}{Ann}

\DeclareMathOperator{\rank}{rank}

\DeclareMathOperator{\tr}{tr}

\DeclareMathOperator{\Ad}{Ad}
\DeclareMathOperator{\ad}{ad}

\DeclareMathOperator{\dv}{div}

\newcommand{\ve}{\varepsilon}

\DeclareMathOperator{\bnabla}{\boldsymbol \nabla}

\numberwithin{equation}{section}

\title[Harmonic maps into sub-Riemannian Lie groups]{Harmonic maps into sub-Riemannian Lie groups}
\author[E. Grong and I. Markina]{Erlend Grong and Irina Markina}

\address{Department of Mathematics, University of Bergen, P.O. Box 7803, 5020 Bergen, Norway}
\email{erlend.grong@uib.no}
\email{irina.markina@uib.no}



\thanks{The first author was by Trond Mohn Foundation-Grant TMS2021STG02 (GeoProCo). The work of both authors was partially supported by the project Pure Mathematics in Norway, funded by the Trond Mohn Foundation and Troms{\o} Research Foundation.}

\begin{document}

\begin{abstract}
We define harmonic maps between sub-Riemannian manifolds by generalizing known definitions for Riemannian manifolds. We establish conditions for when a horizontal map into a Lie group with a left-invariant metric structure is a harmonic map. We show that sub-Riemannian harmonic maps can be abnormal or normal, just as sub-Riemannian geodesics. We illustrate our study by presenting the equations for harmonic maps into the Heisenberg group.
\end{abstract}

\maketitle

\section{Introduction}

A harmonic map between Riemannian manifolds $(M,g)$, $\dim(M)=m$, and $(N,h)$, $\dim(N)=n$, are smooth maps giving the minimum to the energy functional 
\begin{equation}\label{eq:energy functional}
\scrE(f)=\int_M e(f) d\mu, \quad e(f)(x) = \sum_{j=1}^n\sum_{i=1}^m h\big( w_j,df(v_i)\big)^2,
\quad f\colon M\to N,
\end{equation}
where $d\mu$ is the Riemannian volume density on $M$, $\{v_i\}_{i=1}^m$ is an orthonormal basis in $T_x M$, and $\{w_j\}_{j=1}^n$ is an orthonormal basis in $T_{f(x)} N$. 
Particular examples are maps $f\colon [a,b]\to N$, describing the Riemannian geodesics in $N$ and harmonic functions $f\colon M \to \mathbb{R}$. Other examples are minimal surfaces. For instance, a minimal surface in $\mathbb{R}^3$ can be seen as a harmonic map $f\colon [a,b]\times [a,b]\to \mathbb R^3$; see, e.g.,~\cite{MR1632}, \cite{MR0036317}, \cite{MR852409} or more recent survey for minimal submanifolds~\cite{MR576752}. The Euler-Lagrange equations of \eqref{eq:energy functional} correspond to the solution of $\tau(f) = 0$, where $\tau(f) = \tr_g \bnabla_{\times} df(\times)$ denotes the tension field of $f$, defined by using an induced connection on $T^*M \otimes f^* TN$ from the Levi-Civita connections on respectively $M$ and $N$. The celebrated result of~\cite{MR164306}
states that any smooth map $f\in C^{\infty}(M, N)$ from a compact Riemannian manifold $M$ to a manifold $N$ of non-positive scalar curvature can be deformed to a harmonic map.

A generalization of this terminology has been suggested for sub-Riemannian manifolds. A sub-Riemannian manifold is a triplet $(M,D,g)$ consisting of a smooth, connected manifold $M$, a subbundle $D$ of the tangent bundle $TM$, and a sub-Riemannian metric $g$ defined only on vectors in $D$. We assume that $D$ is bracket-generating, meaning that sections of $D$ and a sufficient number of their Lie brackets span $T_xM$ at each point $x\in M$. Studies of harmonic maps $f\colon M\to N$ from a sub-Riemannian $(M,D,g)$ into a Riemannian manifold $(N,h)$ of non-positive curvature was made, for instance, in~\cite{MR1871387,MR4236537,MR3223552,MR1433120}. Here, the energy functional \eqref{eq:energy functional} is modified by letting $v_1, \dots, v_m$ be an orthonormal basis of $D_x$, and the corresponding equation $\tau(f)=0$ turns to be of a hypoelliptic type. The existence and regularity of the solution to $\tau(f)=0$ was established in~\cite{MR4236537} under some convexity condition on $N$, and uniqueness has been shown in~\cite{MR1728089}.

In the present paper we consider harmonic maps allowing the target space to be a sub-Riemannian manifold. Already the study of curves in sub-Riemannian manifolds shows that it is not sufficient to deal exclusively with the Euler-Lagrange equation when it comes to minimizers of~\eqref{eq:energy functional}. More precisely, there are examples of curves that are energy minimizers, and hence the length minimizers, which are not solutions to the Euler-Lagrange equation. Such curves are necessarily singular points in the space of curves of finite sub-Riemannian length, also called horizontal curves, fixing two given points. Minimizers that are solutions of the Euler-Lagrange equation are called normal, and they are smooth~\cite{MR1078163,MR862049}. There are several open questions related to the regularity of minimizers which are singular curves~\cite{MR3932956}; see also~\cite{MR3971262}.  To simplify the exposition we choose the target sub-Riemannian manifold $(N,E,h)$ to be a Lie group with a left-invariant sub-Riemannian structure $(E,h)$; see also~\cite{MR4236551}, where the target space is a Carnot group. The restriction of the target space to a Lie group allows one to avoid some of the complications of $L^2$ and Sobolev maps between general manifolds; see, e.g., \cite[Section~4]{SU83}. Furthermore, applying the Maurer-Cartan form on a Lie group simplifies calculations and prevents the need of to choose an explicit connection for the target manifold as well.
The map $f\colon M\to N$ is required to be horizontal, that is $df(D)\subset E$. We consider the harmonic maps to be analogous of ``normal" and singular geodesics, based on the study of the maps that are regular or singular points of an analogue of the end-point-map. We finally produce equations for both types of horizontal maps: the singular (or abnormal) maps and the normal, latter being solutions of the Euler-Lagrange equation. We will not address conditions for existence or non-existence of such harmonic maps, rather leaving such questions for future research.

We emphasize that we consider a sub-Riemannian analogue of \eqref{eq:energy functional} which is only defined for horizontal maps and a map $f$ is harmonic if it is a critical value under \emph{horizontal variations}. See \eqref{eq:SREnergy} for the definition of the sub-Riemannian energy functional. Such an approach can be considered as the limiting case when the length of vectors outside of $E$ in $TN$ approach infinity. This is in contrast to work in~\cite[Proposition 5.1]{MR2480662} on CR manifolds, which uses an orthogonal projection to define an energy functional for all maps, and where maps are considered harmonic if it is a critical value relative to \emph{all variations}. The latter can be considered as a limiting case where the length of vectors orthogonal to $E$ in $TN$ approach zero. However, we note that if a map $f$ is horizontal \emph{and} harmonic in the sense of the definition in~\cite{MR2480662}, then $f$ will also be harmonic according to our definition, as being critical under all variations implies that $f$ is also critical with respect to horizontal variations.

The structure of the paper is as follows. In Section~\ref{sec:SRGeometry} we introduce sub-Riemannian manifolds, sub-Riemannian measure spaces, and connections compatible with such structures. In Section~\ref{sec:Horizontal}, we define horizontal maps from a compact sub-Riemannian measure space into a Lie group with a left-invariant sub-Riemannian structure, and we show the Hilbert manifold structure of the space of these maps. 
For the rest of the paper, we use the the convention that $M$ is compact, which ensures that the functional in \eqref{eq:energy functional} is finite. Similar to what is done for Riemannian harmonic maps (see, e.g., \cite[Section~2]{EL88}) the case of $M$ non-compact can be considered by calling $f$ harmonic if it is a critical value of the energy functional when restricted to any (relatively) compact subdomain. For simplicity, we will also assume that $M$ is simply connected. See Remark~\ref{re:NSC} where we suggest modifications for a non-simply connected manifold. We introduce the idea of regular and singular maps, as well as some conditions for these maps. Finally, in Section~\ref{sec:Harmonic}, we find equations for both the normal and abnormal harmonic maps. We show that these equations are a natural generalization of above-mentioned cases of maps into Riemannian manifolds, as well as abnormal and normal sub-Riemannian geodesics. We also give an explicit differential equation for harmonic maps into the Heisenberg group.

\section{Sub-Riemannian geometry} \label{sec:SRGeometry}
\subsection{Sub-Riemannnian measure space}

A sub-Riemannian manifold is a triple $(M, D, g)$ where $M$ is a connected manifold, $D$ is a subbundle of $TM$ and $g = \langle \cdot \,, \cdot \rangle_g$ is a metric tensor defined on sections of $D$. Throughout the paper, unless otherwise stated, the subbundle $D$ is assumed to be \emph{bracket-generating}, meaning that the sections of $D$ and their iterated brackets span the tangent space at each point of~$M$. This condition is sufficient to ensure that any pair of points $x_0$ and $x_1$ in $M$ can be connected by \emph{a horizontal curve} $\gamma$, i.e., an absolutely continuous curve such that $\dot \gamma(t) \in D_{\gamma(t)}$ for almost every~$t$; see~\cite{Cho39,Ras38}. Thus, the distance
\begin{equation}\label{eq:CCdistance}
d_g(x_0, x_1) = \inf \left\{ \int_0^1 |\dot \gamma(t)|_{g} \, dt \, : \, \begin{array}{c} \text{$\gamma$ is horizontal} \\ \gamma(0) = x_0, \quad \gamma(1) = x_1 \end{array} \right\}
\end{equation}
is well defined. Furthermore, the metric topology with respect to $d_g$ coincides with the manifold topology on $M$. We do not exclude the possibility $D = TM$.

Associated with the sub-Riemannian metric $g$, there is a vector bundle morphism
\begin{equation}\label{eq:sharp}
\sharp^g: T^*M \to D,
\end{equation}
defined by
$$\alpha(v) = \langle \sharp^g \alpha, v \rangle_{g}$$
for any $x \in M$, $\alpha \in T_x^*M$, and $v \in D_x$.
Define a cometric $g^* = \langle \cdot \,, \cdot \rangle_{g^*}$ on $T^*M$ by
$$\langle \alpha, \beta \rangle_{g^*} = \alpha(\sharp^g \beta)=\langle \sharp^g \alpha, \sharp^g\beta\rangle_g, \qquad \alpha, \beta\in \Gamma(T^*M).$$
This cometric is exactly degenerated along the subbundle $\Ann(D) \subseteq T^*M$ of covectors vanishing on $D$. Conversely, given a cometric $g^*$ on $T^*M$ that is degenerated along a subbundle of $T^*M$, we can define the subbundle $D$ of $TM$ as the image of the map $\sharp^g: \alpha \mapsto \langle \alpha, \cdot \rangle_{g^*}$ in~\eqref{eq:sharp}, and a metric $g$ on $D$ by the relation
$$\langle \sharp^g \alpha, \sharp^g \beta \rangle_{g} = \langle \alpha, \beta \rangle_{g^*}.$$
Hence, a sub-Riemannian manifold can equivalently be defined as a connected manifold with a symmetric positive semi-definite tensor $g^*$ on $\Gamma(TM^{\otimes 2})$ degenerating on a subbundle of $T^*M$. In what follows, we will speak about a sub-Riemannian structure interchangeably as $(D,g)$ or $g^*$, assuming that the subbundle $D$ is bracket-generating. For more on sub-Riemannian manifolds, see, e.g., \cite{MR3971262,Mon02}.

\begin{definition}
A sub-Riemannian measure space $(M, D, g, d\mu)$ is a sub-\linebreak Riemannian manifold $(M,D, g)$ with a choice of smooth volume density $d\mu$ on $M$. If $D = TM$, then $d\mu$ is the volume density of the Riemannian metric $g$.  
\end{definition}
On a sub-Riemannian measure space $(M, D, g, d\mu)$ there is a unique choice of second order operator
\begin{equation}\label{eq:LaplaceBeltrami}
\Delta_{g,d\mu} \phi = \dv_{d\mu} \, \sharp^g d\phi,\qquad \phi\in C^{\infty}(M).
\end{equation}
We call the operator in~\eqref{eq:LaplaceBeltrami}~ \emph{the sub-Laplacian} of the sub-Riemannian measure space. Since $D$ is bracket-generating, the classical result of H\"ormander~\cite{Hor67} states that $\Delta_{g,d\mu}$ is a hypoelliptic operator. If the measure $d\mu$ is clear from the context, we simply write $\Delta_g$. We also denote the sub-Riemannian measure space as $(M, g^*,d\mu)$.

We say that a Riemannian metric $\bar{g} = \langle \cdot \,, \cdot \rangle_{\bar{g}}$ is \emph{a taming metric} of $(M, D, g, d\mu)$ if $g$ is the restriction $\bar g|D$ of $\bar{g}$ to $D$ and the volume density of $\bar{g}$ equals $d\mu$. 
\begin{lemma}
Any sub-Riemannian measure space has a taming Riemannian metric.
\end{lemma}
\begin{proof}
Let $(M, D, g, d\mu)$ be any sub-Riemannian measure space with $\dim M = m$ and $\rank D = k$. If $k =m$, then by convention $\bar{g}= g$ is a taming Riemannian metric. For $k<m$, we take an arbitrary Riemannian metric $\bar{g}_0$ on $M$ and let $D^\perp$ denote the orthogonal complement of $D$ with respect to $\bar{g}_0$. The rank of $D^\perp$ equals $m-k$. Define a Riemannian metric $\bar{g}_1$ such that $D$ and $D^\perp$ are still orthogonal with respect to $\bar{g}_1$, and $\bar{g}_1|D = g$, $\bar{g}_1|D^{\perp} = \bar{g}_0|D^\perp$. Let $d\bar{\mu}$ be the Riemannian volume density with respect to $\bar{g}_1$, and write $d\bar{\mu} = \rho d\mu$. 
 
Finally, we define the metric $\bar{g}$ to be such that $D$ and $D^\perp$ are orthogonal with respect to $\bar{g}$ and
$$\bar{g}|D =g, \qquad \bar{g}|D^\perp = \rho^{-1/(m-k)}\bar{g}_1|D^{\perp},$$
which gives us the desired Riemannian metric.
\end{proof}

\begin{remark}
[Hausdorff and Popp's measure] A manifold $M$ carries a measure $dx$ which is the pushforward of the Lebesgue measure by the chart map. 
The distance $d_g$ in~\eqref{eq:CCdistance} generated by the sub-Riemannian metric tensor $g$ produces the Hausdorff measure $dH$. Relative to any coordinate system defined sufficiently close to a regular point, $dH = q(x) dx$ is absolutely continuous with respect to $dx$. It is not clear whether $q$ is a smooth function. Another construction of a measure near regular point has been provided by O. Popp (see \cite[Chapter 10]{Mon02}) which gives a measure $d\mu$ with a smooth Radon-Nikodym derivative with respect to $dx$. The latter allows one to define the sub-Laplacian by making use of the integration by parts with respect to the smooth measure $d\mu$, which leads to the sub-Laplacian introduced in~\cite{ABGR09}. For the case of the Carnot groups, both the Hausdorff and the Haar measures are equal up to a constant, and are hence all smooth.
\end{remark}

\subsection{Compatible connections on a sub-Riemannian measure space}

Consider a sub-Riemannian structure $g^*$ on $M$. For a two-tensor field $\xi \in \Gamma(T^*M^{\otimes 2})$ we write
$$\tr_g \xi(\times, \times ) = \xi(g^*),\quad \text{i.e.,}\quad \tr_g \xi(\times, \times)(x)  =\sum_{i=1}^k \xi(v_i, v_i)$$
for an arbitrary orthonormal basis $v_1, \dots, v_k$ of $D_x$ with $k=\rank D$. We want to consider connections on sub-Riemannian manifolds and sub-Riemannian measure spaces. We begin with the following definition of a connection on tensor fields; see, for instance,~\cite[Chapter 4]{MR1468735}.
\begin{definition}
Let $\nabla$ be an affine connection on $TM$.
\begin{enumerate}[\rm (a)]
\item We say that $\nabla$ is compatible with $(D,g)$ $($equiv. $g^*$$)$ if it satisfies the following equivalent conditions:
\begin{enumerate}[\rm (i)]
\item $\nabla g^* = 0$,
\item $\nabla \sharp^g = \sharp^g \nabla$,
\item For any $X_1, X_2 \in \Gamma(D)$, $Z \in\Gamma(TM)$, we have that $\nabla_Z X_1 \in \Gamma(D)$ and
$$Z \langle X_1, X_2 \rangle_g = \langle \nabla_{Z} X_1, X_2 \rangle_g + \langle X_1, \nabla_Z X_2 \rangle_g \, .$$ 
\end{enumerate}
\item We say that $\nabla$ is compatible with $(D,g, d\mu)$ $($equiv.  $(g^*, d\mu)$$)$ if $\nabla$ is compatible with $(D,g)$ $($equiv. $g^*$$)$ and for any $\phi \in C^\infty(M)$
$$\tr_g \nabla_{\times} d\phi(\times) = \sum_{i=1}^k \nabla_{v_i} d\phi(v_i)=\Delta_{g, d\mu}\phi$$
for an orthonormal basis $v_1, \dots, v_k$ of $D_x$.
\end{enumerate}
\end{definition}

The following is known on sub-Riemannian manifolds.
\begin{proposition}
\cite{GrTh16} Let $g^*$ be a sub-Riemannian structure and $d\mu$ a volume density on $M$. Then $(g^*, d\mu)$ has a compatible connection.
\end{proposition}
We also prove the following result.
\begin{lemma} \label{lemma:Basis}
\begin{enumerate}[\rm (a)]
\item A connection $\nabla$ is compatible with $g^*$ if and only if for every point $x\in M$ there exists a local orthonormal frame $X_1,\dots, X_k$ of $D$ around $x$ such that $\nabla X_j (x) = 0$.
\item A connection $\nabla$ is compatible with $(g^*, d\mu)$ if and only if for every point $x\in M$ there exists a local orthonormal frame $X_1,\dots ,X_k$ of $D$ around $x$ such that $\nabla X_j (x) = 0$ and $\dv_{d\mu} X_i(x) = 0$.
\end{enumerate}
\end{lemma}
\begin{proof}
If $\nabla$ preserves $D$, then $\nabla|D$ is a connection on $D$ preserving the inner product $g$. Hence there is a local orthonormal frame of $D$ that is parallel with respect to $\nabla$ at a given point $x$; see, e.g., \cite[Theorem~2.1 and Remark~2.2]{Gro22} for details. Conversely, let $\alpha$ be an arbitrary one-form and $x \in M$ an arbitrary point. Assume that there exists an orthonormal frame $X_1, \dots, X_k$ of $D$ around $x$ such that it is $\nabla$-parallel at $x$. Completing calculations at $x$, we obtain
$$X | \alpha|^2_{g^*}(x) = \sum_{i=1}^k X (\alpha (X_i))^2(x) =2  \sum_{i=1}^k \alpha(X_i)(x) (\nabla_X \alpha)(X_i)(x) = 2\langle \alpha, \nabla_X \alpha \rangle_{g^*}(x).$$
If we can find such a basis for every point in $M$, it follows that $\nabla$ is compatible with $g^*$. This proves (a).

The result in (b) follows from the identity
$$\Delta_{g,d\mu} f = \sum_{i=1}^k X_i^2 f + \sum_{i=1}^k (X_if) \dv_{d\mu} X_i,$$
that holds for any local orthonormal basis of $M$.
\end{proof}

\begin{corollary}
Let $\nabla$ be a connection compatible with $(g^*, \mu)$ and let $X$ be a horizontal vector field. Then
$$\dv_{d\mu} X = \tr_g \langle \nabla_\times X, \times \rangle_g.$$
\end{corollary}
\begin{proof}
For a given point $x \in M$, choose an orthonormal frame $X_1, \dots, X_k$ of $D$ around $x$ with $\nabla X_i(x) = 0$ and $\dv_{d\mu} X_i(x) = 0$. Write $X = \sum_{i=1}^k f_i X_i$. Then
\begin{eqnarray*}
\dv_{d\mu} X(x) &=& 
\sum_{i=1}^k \big(X_i f_i(x) + f_i(x) \dv_{d\mu}X_i(x)\big) \\
&=& \sum_{i=1}^k X_i f_i(x) 
=
\sum_{i=1}^k \langle \nabla_{X_i} X, X_i \rangle_g(x).
\end{eqnarray*}
Since $x \in M$ was arbitrary, the result follows.
\end{proof}

\subsection{Left-invariant sub-Riemannian structures}\label{sec:LIstructure}

Let $G$ be a Lie group with Lie algebra $\mathfrak{g}$. Let $(E,h)$ be a sub-Riemannian structure on $G$. We say that the sub-Riemannian structure is left-invariant if $E$ is a left-invariant distribution and if
$$\langle v, w \rangle_h = \langle a\cdot v, a \cdot w \rangle_{h}, \qquad \text{for any $a \in G$,\ \  $v,w \in E_1=\mathfrak e\subseteq \mathfrak{g}$},$$
where we denote by $a\cdot v$ the action on $v\in \mathfrak{g}$ by the differential of the left translation by $a\in G$, and $1 \in G$ is the identity element. Equivalently, let $\omega$ be \emph{the left Maurer-Cartan form}, given by $\omega(v) = a^{-1} \cdot v \in \mathfrak{g}$ for any $v \in T_a G$. Then $v\in E$ if and only if $\omega(v) \in \mathfrak{e} = E_1 \subseteq \mathfrak{g}$. We then say that $(E,h)$ is obtained by left translation of $(\mathfrak{e},\langle \, \cdot \,, \cdot\rangle)$.

\begin{example}[The Heisenberg group] \label{ex:Hn}
We consider the space $H^n = \mathbb{R}^{2n+1}$ with coordinates $(a,b,c) =  (a_1, \dots, a_n, b_1, \dots, b_n,c)$. We give this space a global frame
\begin{equation}\label{eq:global_frame}
A_j = \partial_{a_j} - \frac{1}{2} b_j \partial_c, \qquad B_j = \partial_{b_j} + \frac{1}{2} a_j \partial_c, \qquad C= \partial_{c}.
\end{equation}
The corresponding coframe is given by $da_j$, $db_j$ and $\theta = dc + \frac{1}{2} \sum_{j=1}^n (b_j da_j - a_j db_j)$. Note the bracket relations
\begin{equation} \label{basisHeis} [C, A_i] = [C, B_j] = [A_i, A_j] = [B_i, B_j] = 0, \qquad [A_i, B_j] = \delta_{ij} C.\end{equation}
Hence, these vector fields form a Lie algebra which we will write as $\mathfrak{h}_n$. We can give $\mathbb{R}^{2n+1}$ a group structure such that the vector fields in~\eqref{eq:global_frame} become left-invariant. The group multiplication is given by
$$(a,b, c) \cdot (\tilde a, \tilde b, \tilde c) = \left(a + \tilde a, b+  \tilde b,  c+\tilde c +\frac{1}{2} \big(\langle a, \tilde b\rangle_{\mathbb{R}^n} - \langle \tilde a, b \rangle_{\mathbb{R}^n} \big) \right). $$
We will define a sub-Riemannian structure $(E,h)$ on $H^n$ by letting $A_1,\ldots,A_n$, $B_1,\ldots,B_n$ be an orthonormal basis.
\end{example}

\section{Horizontal maps into Lie groups} \label{sec:Horizontal}

\subsection{Maps into Lie groups and the Darboux derivative} 

In what follows, we will let $\Omega^p(M,\mathfrak g)$ be the space of $\mathfrak g$-valued differential $p$-forms on a manifold $M$. 
We recall the definition and properties of the Darboux derivative; referring to \cite{Sha97} for more details. Let $G$ be a Lie group with Lie algebra~$\mathfrak{g}$. Let $\omega \in \Omega^1(G, \mathfrak{g})$ be the left Maurer-Cartan form as defined in Section~\ref{sec:LIstructure}. This form satisfies the left Maurer-Cartan equation
\begin{equation} \label{LeftMC} 
d\omega + \frac{1}{2} [\omega, \omega] = 0,
\end{equation}
with $[\omega,\omega]$ being the two-form $(v,w) \mapsto 2[\omega(v), \omega(w)]$. See Appendix~\ref{sec:Forms} for more details.
If $M$ is a given manifold and $f \colon M \to G$ is a smooth map, we say that $\alpha_f : = f^* \omega$ is the \emph{left Darboux derivative} of $f$. It follow from definition that $\alpha_f$ satisfies \eqref{LeftMC}. Conversely, if $\beta \in \Omega^1(M, \mathfrak{g})$ satisfies $d\beta + \frac{1}{2} [\beta, \beta] = 0$, then locally $\beta$ is the Darboux derivative of some function. If the monodromy representation of $\beta$ is trivial (see~\cite[Chapter 3, Theorem 7.14]{Sha97}) then the structural equation implies that $\beta=\alpha_f$ for some map $f\colon M \to G$. Particularly, for a connected, simply connected manifold $M$ the monodromy representation of any $\mathfrak g$-valued one-form is trivial, meaning that any form satisfying the left Maurer-Cartan equation can be represented as a Darboux derivative. Through the rest of the paper, we assume that $M$ is connected and simply connected.

Denote by $\calA \subseteq \Omega^1(M,\mathfrak{g})$ the collection of forms $\alpha$ satisfying $d\alpha + \frac{1}{2} [\alpha, \alpha] =0$, and define
$$T_\alpha\calA = \left\{ \dot \beta(0) \, : \, \begin{array}{c} \beta:(-\ve,\ve) \to \Omega^{1}(M, \mathfrak{g}) \text{ is smooth,} \\
\beta(0)= \alpha, \beta(t) \in \calA \text{ for any $t \in (-\ve,\ve)$.} \end{array}\right\}.$$
\begin{lemma} \label{lemma:tangent}
We have
$$T_\alpha \mathcal{A} = \left\{ dF + [\alpha, F] \, : \, F \in C^\infty(M, \mathfrak{g}) \right\}.$$
\end{lemma}

\begin{proof}
Let $\beta(t)$ be a differentiable curve in $\Omega^1(M,\mathfrak{g})$ and assume that $d\beta(t) + \frac{1}{2} [\beta(t), \beta(t)] =0$. If we differentiate this relation and assume that $\beta(0)=\alpha$ and $\dot \beta(0)=\eta$, then
$$d\eta + [\alpha, \eta] =0.$$
If $\alpha = f^* \omega$ and we write $\eta = \Ad(f^{-1}) \tilde \eta$, then
\begin{equation} \label{FoundClosed} \Ad(f^{-1}) d \tilde \eta = 0.\end{equation}
We remark that here we are abusing notation to write $\eta|_x = \Ad(f(x)^{-1}) \tilde \eta|_x$, where $f(x)^{-1}$ is the inverse of $f(x)$ with respect to the group operation in $G$. To see that \eqref{FoundClosed} holds, recall first that for any curve $A(t)$ in a Lie group $G$, we have that $\frac{d}{dt} \Ad(A(t)^{-1}) =- \ad(\omega(\dot A(t))) \Ad(A(t)^{-1})$, where $\omega$ is the Maurer-Cartan form for the group $G$. Considering the special case where $A(t) = f(\gamma(t))$ for an arbitrary smooth curve $\gamma(t)$ in $M$ with $f^* \omega=\alpha$, we get the formula for the differential
$$d\Ad(f^{-1}) = - \ad(f^*\omega) \Ad(f^{-1}) = - \ad(\alpha) \Ad(f^{-1}).$$
Using the definition of $\eta= \Ad(f^{-1}) \tilde \eta$, this leads to
$$0 = d\eta + [\alpha, \eta] = (d\Ad(f^{-1})) \wedge \tilde \eta + \Ad(f^{-1}) d\tilde \eta + [\alpha, \Ad(f^{-1}) \tilde \eta] =  \Ad(f^{-1}) d\tilde \eta.$$

In summary, the form $\tilde \eta \in \Omega^1(M, \mathfrak{g})$ is closed and we can find a function $\tilde F\colon M \to \mathfrak{g}$ such that $\tilde \eta = d\tilde F$ due to the vanishing de Rham cohomology; see~\cite[Theorem 11.14]{MR1930091}. Furthermore, if we define $F = \Ad(f^{-1}) \tilde F$, then
$$\beta = \Ad(f^{-1}) d\tilde F = dF + [\alpha,F].$$

Conversely, for any $F \in C^\infty(M, \mathfrak{g})$, we can define a curve $g(t) = f\cdot \exp({tF})$ in the space of smooth maps $M\to G$ and $\beta(t) = g(t)^* \omega$. Here $\exp\colon \mathfrak g\to G$ is the group exponential. Let $v \in T_xM$, $x\in M$, be arbitrary and define $\gamma(s)\colon(-\epsilon, \epsilon)\to M$, $\epsilon>0$, as a curve with $\gamma(0)=x$ and $\partial_s\gamma(0) = v$. If we set $\Gamma(s,t) = g(t) (\gamma(s))$, then we compute 
\begin{align*}
& \dot \beta(0)(v)  = \partial_t \omega( \partial_s \Gamma(s,t)) |_{(s,t) = (0,0)} = \partial_t \Gamma^*\omega(\partial_s) |_{(s,t) = (0,0)}
\\
&=  \Big(\partial_s \Gamma^* \omega(\partial_t) - d(\Gamma^* \omega)(\partial_s , \partial_t)\Big)|_{(s,t)= (0,0) } \\
&=\Big(\partial_s\big(\omega(\partial_t\Gamma)\big)- d(\Gamma^* \omega)(\partial_s , \partial_t)\Big)|_{(s,t)= (0,0)} 
\\
& =  \partial_s F\big(\gamma(s)\big)|_{s=0} + [\omega( \partial_s \Gamma), \omega(\partial_t \Gamma )]|_{(s,t) = (0,0)} \\
& = dF(v) + [\alpha(v), F(x)].
\end{align*}
Recall that  $\alpha = f^* \omega$. The result follows.
\end{proof}

We want to close our space of Darboux derivatives into a Hilbert space. Let $(M, D,g, d\mu)$ be a sub-Riemannian measure space and let $\bar{g}$ be a taming Riemannian metric.
Extend the inner product on $\mathfrak{e}$ to a full inner product on $\mathfrak{g}$. These choices give us an induced inner product on $\wedge^k T^*M \otimes \mathfrak{g}$, which allows us to define an $L2$-inner product $\langle \beta, \beta \rangle = \int_M \langle \beta(x), \beta(x) \rangle \, d\mu(x)$ for any $\beta \in \Omega^k(M,\mathfrak{g})$. With this definition, we consider $L^2\Omega(M,\mathfrak{g}) = \oplus_{k=0}^{\dim M} L^2\Omega^k(M,\mathfrak{g})$ as the space of $L^2$-forms with values in $\mathfrak{g}$. We remark that since $M$ is compact and $\mathfrak{g}$ is finite dimensional, any other choice of taming Riemannian metric $\bar{g}$ and inner product of $\mathfrak{g}$ will give us an equivalent $L^2$-inner product, meaning in particular that $L^2$-forms are independent of these choices.
More about the theory of $L^p$ forms can be found, for instance in~\cite{MR1297538}.

\begin{corollary} \label{cor:Hilbert}
Assume that $M$ is simply connected and compact. Then the closure $\overline{\calA}$ of $\calA$ in $L^2\Omega^1(M,\mathfrak{g})$ is a Hilbert submanifold of $L^2\Omega^1(M,\mathfrak{g})$ with tangent space $\overline{T_\alpha \mathcal{A}} \subseteq L^2 \Omega^1(M, \mathfrak{g})$.
\end{corollary}
\begin{proof}
Let $\alpha \in \calA$ be the Darboux derivative $\alpha = \alpha_f = f^* \omega$ of a map $f$. We consider an arbitrary curve $\beta \in \mathcal{A}$ that can be written as $\beta  = dF + [\alpha, F]$ for some $F \in C^\infty(M,\mathfrak{g})$. From the proof of Lemma~\ref{lemma:tangent}, we note that if $F = \Ad(f^{-1}) \tilde F$, then $\Ad(f^{-1}) \beta = d\tilde F$. We denote by $\tilde F_\beta$ a unique solution to this equation satisfying $\int_M \tilde F_\beta d\mu = 0$, and we define
$F_\beta = \Ad(f^{-1}) \tilde F_\beta$. Then,
\begin{align*} 
\| F_\beta\|_{L^2} & \leq \|\Ad(f^{-1})\|_{L^\infty} \| \tilde F_\beta\|_{L^2} 
\\
&  \stackrel{\text{Poincar\'e}}{\leq} C \|\Ad(f^{-1})\|_{L^\infty}\| d\tilde F_\beta\|_{L^2} \leq  C \|\Ad(f^{-1})\|_{L^\infty} \|\Ad(f)\|_{L^\infty} \| \beta \|_{L^2} 
\end{align*}
for some constant $C>0$. Note that the linear map $\beta \mapsto \tilde F_\beta$ is bounded and invertible with respect to the $L^2$ metric, which, in particular, is smooth. Here we have used the Poincar\'e inequality for compact Riemannian manifolds found in, e.g., \cite[Theorem~2.10]{Heb99}. It follows that this map can be extended by limits to be well defined as a map from $\overline{T_\alpha \calA}$ to $\{ F \in L^2(M, \mathfrak{g}) \, : \, \int_M \Ad(f) F \, d\mu =0 \}$.

Continuing, we introduce a map $\Phi \colon T_\alpha \calA \mapsto \calA$ as
$$\Phi(\beta) = (f \cdot e^{F_\beta})^* \omega,\quad \text{for}\quad\beta\in T_\alpha \calA.$$
We observe then that for any $v \in T_xM$, we can apply the formula for the differential of the Lie group exponential to obtain
\begin{align*}
\Phi(\beta)(v) & = e^{- F_{\beta} (x)}   f(x)^{-1} df(v)  e^{F_{\beta}(x)}
+  \frac{1- e^{-\ad(F_{\beta}(x))}}{\ad(F_{\beta}(x))} dF_{\beta}(v)  
\\
& = e^{- \ad(F_{\beta}(x))} \alpha 
+ \frac{1- e^{-\ad(F_{\beta}(x))}}{\ad(F_{\beta}(x))} \big(\beta(v) + \ad(F_\beta(x)) \alpha(v)\big) 
\\
& = \alpha(v) + \frac{1- e^{-\ad(F_{\beta}(x))}}{\ad(F_{\beta}(x))} \beta(v),
\end{align*} 
meaning that
\begin{align*}
\Phi(\beta) & = \alpha + \frac{1- e^{-\ad(F_{\beta})}}{\ad(F_{\beta})} \beta = \alpha + \sum_{n=0}^\infty \frac{(-1)^n \ad(F_\beta)^n}{(n+1)!} \beta.
\end{align*} 
This map is well defined for any $\beta \in \overline{T_\alpha \calA}$, giving a smooth map $\Phi:\overline{T_\alpha \calA} \to \overline{\calA} \subseteq L^2 \Omega^1(M, \mathfrak{g})$.
Furthermore, we see that its Fr\'echet differential at $\beta =0$ is given by
$$D\Phi|_0(\beta) = \beta.$$
Thus, $\Phi$ is locally injective, so it can be used as a chart close to $0 \in \overline{T_\alpha \calA}$. Since $\calA$ is dense in $\overline{\calA}$ the result follows.
\end{proof}

\subsection{Horizontal maps}
For the rest of this section, $(M,D,g,d\mu)$ will be a simply connected, compact sub-Riemannian measure space while $G$ will be a Lie group with Lie algebra $\mathfrak{g}$ and left Maurer-Cartan form $\omega$. The structure $(E,h)$ on $G$ will be defined by left translation of $(\mathfrak{e}, \langle\cdot\, , \cdot\rangle )$.
We introduce the following concept.
\begin{definition}
Let $(M, D, g)$ and $(N, E, h)$ be two sub-Riemannian manifolds. We say that a smooth map $f\colon M \to N$ is \emph{horizontal} if $df(D) \subseteq E$.
\end{definition}
To simplify the discussion in this paper, we only consider the case when $N = G$ is a Lie group $G$ with a left-invariant sub-Riemannian structure $(E,g)$ that is the left translation of a vector space $\mathfrak{e}\subset\mathfrak g$ and a scalar product $\langle \cdot \,, \cdot \rangle$ on $\mathfrak{e}$. Then, $f\colon M \to G$ is horizontal if and only if $\alpha_f = f^* \omega$ sends $D$ into $\mathfrak{e}$. We write $\calA_{D,E}$ for the collection of such forms $\alpha_f$. 

Consider $\Omega^1(D, V) = \Gamma(D^* \otimes V)$ as partial one-forms only defined on $D$ with values in a vector space $V$. Write $L^2 \Omega^1(D,V)$ for its $L^2$-closure. Consider $\overline{\calA_{D,E}} \subseteq \overline{\calA}$. Define a linear map 
\begin{equation}\label{eq:P}
P\colon L^2 \Omega^1(M, \mathfrak{g}) \to L^2\Omega^1(D, \mathfrak{g}/\mathfrak{e})\quad
\text{by}\quad
P(\alpha) = \alpha|D \mod \mathfrak{e}.
\end{equation}
Then $\overline{\calA_{D,E}} = \ker P \cap \overline{\calA}$. 

\begin{definition}
We say that $\alpha \in \overline{\calA_{D,E}}$ is regular $($respectively, singular$)$ if $\alpha$ is a regular $($respectively, singular$)$ point of $P|_{\calA}$; that is the differential of the map $P$ is surjective (not surjective) at $\alpha\in\mathcal A$. We say that a sub-Riemannian horizontal map $f\colon M \to G$ is regular $($respectively, singular$)$, if its Darboux derivative $\alpha_f \in \calA_D$ is regular $($respectively, singular$)$.
\end{definition}
Since $\overline{\mathcal{A}_{D,E}}= \ker P \cap \overline{\calA}$, the implicit function theorem implies that $\overline{\mathcal{A}_{D,E}}$ has the local structure of a manifold around any regular $\alpha$.

\subsection{Strong bracket generating condition}
We list the conditions for distributions on $M$ and $G$, which guarantee the absence of singular morphisms. 
\begin{definition}
We say that $\mathfrak{e}\subset\mathfrak g$ is a strongly $q$-bracket generating subspace of $\mathfrak g$ if for any $1 \leq l \leq q$ and any set of linearly independent vectors $A_1, \dots, A_l \in \mathfrak{e}$ and any $Z_1, \dots, Z_l\in \mathfrak{g}$, there exists an element $B \in \mathfrak{e}$ such that
$$Z_j - [A_j, B] \in \mathfrak{e}, \qquad j=1, \dots, l.$$
\end{definition}

\begin{example} \label{ex:sb}
Let $\theta$ be a left-invariant one-form on a $(2n+1)$-dimensional Lie group $G$ and define $\ker \theta|_1 = \mathfrak{e} \subseteq \mathfrak{g}$. Assume that $d\theta|(\wedge^2 \mathfrak{e})$ is non-degenerate, i.e., $\theta$ is a contact form on $G$. We can find a basis $A_1, \dots, A_n, B_1, \dots, B_{n}$ of $\mathfrak{e}$ such that $d\theta(A_i, A_j) = d\theta(B_i, B_j) = 0$ and $d\theta(B_i, A_j) = \delta_{ij}$. Let $Z \in \mathfrak{g}$ be the unique element satisfying $\theta(Z) = 1$ and $d\theta(Z, \, \cdot \,) = 0$. To find $B\in\mathfrak e$, we need to solve the equations
$$[p_j A_i, B]= \tilde p_j Z,\qquad [q_j B_i, B]= \tilde q_j Z, \qquad p_j \neq 0,\quad q_j \neq 0.$$
One can easily check that
$$B = \sum_{i=1}^n \left( \frac{\tilde p_i}{p_i} B_i -\frac{\tilde q_i}{q_i} A_i\right).$$
is a solution. This shows that such structures are strongly $2n$-bracket generating.
In particular, we note that the Heisenberg group $H^n$ is has a strong $2n$-bracket generating distribution.
\end{example}

\begin{proposition} \label{prop:sb}
Let $(M,D,g)$ be a sub-Riemannian manifold, where $M$ is simply connected and $D$ has rank $k \geq 2$. Let $\mathfrak{g}$ be a Lie algebra with a generating subspace $\mathfrak{e} \subseteq \mathfrak{g}$ of positive codimension.
Let $\alpha \in \Omega^1(M, \mathfrak{g})$ be a one-form satisfying $\alpha(D) \subseteq \mathfrak{e}$.
\begin{enumerate}[\rm (a)]
\item Assume that there exists a non-intersecting horizontal loop $\gamma:[0,1] \to M$ such that $\alpha(\dot \gamma(t)) = 0$ for almost every $t \in [0,1]$. Then $\alpha$ is singular.
\item Assume that $\mathfrak{e} \subseteq \mathfrak{g}$ is strongly $k$-bracket generating. If $\alpha|D$ is injective at every point, then $\alpha$ is regular.
\end{enumerate}
\end{proposition}
\begin{proof}
Choose a complement $\mathfrak{f}$ to $\mathfrak{e}$ in $\mathfrak{g}$.
Let $F \in C^\infty(M, \mathfrak{g})$ be a function and write $F = F_{\mathfrak{e}} + F_{\mathfrak{f}}$ according to the decomposition $\mathfrak{g} = \mathfrak{e} \oplus \mathfrak{f}$. Recall that the regularity of $\alpha$  is equivalent to the assumption that for any one-form $\psi\in \Omega^1(M, \mathfrak{g})$, one can choose $F_{\mathfrak{e}}$ and $F_{\mathfrak{f}}$ such that
\begin{equation}\label{eq:psi}
dF_{\mathfrak{f}}|D + [\alpha |D, F_{\mathfrak{e}}] + [\alpha |D, F_{\mathfrak{f}}] = \psi|D \mod \mathfrak{e}.
\end{equation}
\begin{enumerate}[\rm (a)]
\item If $\gamma: [0,1] \to M$ is a non-intersecting horizontal loop, then $x_1 = \gamma(1/2) \neq x_0 = \gamma(0) = \gamma(1)$ by assumption. Define $\gamma_1, \gamma_2: [0,1] \to M$ by $\gamma_1(t) = \gamma(t/2)$ and $\gamma_2(t) = \gamma(1- t/2)$, which are non-intersecting horizontal curves from  $x_0$ to $x_1$. Let $U$ be an open set that does not intersect $\gamma_2$, but intersects with a subset of $\gamma_1$ of positive length. Let $\psi = \psi_0 \otimes Z$, where $Z \in \mathfrak{f}$, $Z\neq 0$, and $\psi_0$ denotes a real valued one-form with support in $U$ such that $C= \int_0^1 \psi_0(\dot \gamma(t)) \,  dt >0$. If we find a function $F = F_{\mathfrak{e}}+ F_{\mathfrak{f}}$ solving \eqref{eq:psi} then
$$F_{\mathfrak{f}}(x_1) - F_{\mathfrak{f}}(x_0) = \int_{\gamma_1} \psi \, = C Z\neq 0.$$ 
However, in order for \eqref{eq:psi} to hold, we would also need to $F_{\mathfrak{f}}(x_1) - F_{\mathfrak{f}}(x_0)$ tp equal $\int_{\gamma_2} \psi$ which is clearly 0 by the definition of $\psi$, giving us a contradiction.
\item If $\alpha|D$ is injective, then we can choose $F_{\mathfrak{f}} =0$. To show the regularity of $\alpha$ we need to solve the equation $[\alpha|D , F_{\mathfrak{e}}] = \psi|D \mod \mathfrak{e}$. The assumption of $D$ being strongly $k$-bracket generating implies that the equation $[\alpha|D , F_{\mathfrak{e}}] = \psi|D \mod \mathfrak{e}$ has a solution for any one-form $\psi\in \Omega^1(M, \mathfrak{g})$. To be more precise, let $X_1$, $\dots$, $X_k$ be a local basis of $D$ and $\psi\in \Omega^1(M, \mathfrak{g})$. We respectively define $A_j \in C^\infty(M, \mathfrak{e})$ and $Z_j \in C^\infty(M, \mathfrak{g})$ by $A_j = \alpha(X_j)$ and $Z_j = \psi(X_j)$, $j=1, \dots, k$. Then we can then define $F_{\mathfrak{e}}$ such that $[ A_j, F_{\mathfrak{e}}] = Z_j$ by the strongly $k$-bracket generating condition on $D$.
\end{enumerate}
\end{proof}

\section{Harmonic maps} \label{sec:Harmonic}
\subsection{Normal and abnormal harmonic maps}
Let $(M, D, g, d\mu)$ be a given sub-Riemannian measure space and let $(G, E,h)$ be a Lie group with a left-invariant sub-Riemannian structure. For a horizontal map $f\colon M \to G$ with Darboux derivative $\alpha_f$, we define its energy as
\begin{align} \label{eq:SREnergy}
\scrE(f) & = \frac{1}{2} \int_M |df |_{g^* \otimes f^*h}^2 d\mu = \frac{1}{2} \int_M \tr_g (f^*h)(\times, \times)\, d\mu \\ \nonumber
&  = \frac{1}{2} \int_M |\alpha_f|^2_{g^*} d\mu =: \hat \scrE(\alpha_f).
 \end{align}
We note that if $v_1, \dots, v_k \in D_x$ and $w_1, \dots, w_n \in E_{f(x)}$ are respective orthonormal bases then
$$|df |_{g^* \otimes f^*h}^2(x) = \sum_{i=1}^k | df(v_i) |_h^2 = \sum_{j=1}^n \sum_{i=1}^k \langle w_j, df(v_i) \rangle_h^2.$$
We would generalize the definition of harmonic maps from~\cite{MR164306} to the sub-Riemannian setting, saying that $f$ is harmonic if it is a critical value of $\scrE$. Instead, we 
use the Darboux derivative to make this definition precise. For $\alpha \in \overline{\calA_{D,E}}$, we define a variation $\alpha_s$ of $\alpha$
as a differentiable curve $(-\ve, \ve) \to \overline{\calA_{D,E}}$, $s \mapsto \alpha_s$, such that $\alpha_0 = \alpha$.
\begin{definition}
We say that $\alpha \in \overline{\calA_{D,E}}$ is harmonic if it is a critical point of $\hat \scrE$, i.e., for every variation $\alpha_s$ of $\alpha$,
we have $\frac{d}{ds} \hat \scrE(\alpha_s)|_{s=0} =0$. We say that $f$ is harmonic if $\alpha_f$ is harmonic. 

\end{definition}
We have the following result. 
\begin{theorem}\label{th:harmonic}
Let $M$ be a simply connected, compact manifold, and $\nabla$ a connection compatible with the sub-Riemannian measure space $(M,g^*,d\mu)$. Let the map $\sharp = \sharp^{h}_1: \mathfrak{g}^* \to \mathfrak{e}$ correspond to the sub-Riemannian metric
$h$ at the identity. Assume that $\alpha \in \calA_{D,E}$ is harmonic. Then at least one of the following statements holds.
\begin{enumerate}[\rm (a)]
\item \emph{(Abnormal case)} There exists form a $\eta \in \Omega^1(M, \mathfrak{g}^*)$ with $\eta|D$ non-zero, satisfying $\sharp \eta|D = 0$ and
$$\delta_D \eta - \tr_{g} \ad^*(\alpha(\times) )\eta(\times) = 0.$$
with $\delta_D \eta = - \tr_{g}  \nabla_\times\eta(\times)$.
\item \emph{(Normal case)} There exists a form $\lambda \in \Omega^1(M, \mathfrak{g}^*)$ satisfying $\sharp \lambda|D = \alpha|D$ and
$$\delta_D \lambda - \tr_{g} \ad^*(\alpha(\times) )\lambda(\times) = 0.$$
with $\delta_D \lambda = - \tr_{g}  \nabla_\times\lambda(\times)$.
\end{enumerate}
\end{theorem}
Recall that $\ad^*$ denotes the adjoint representation of $\mathfrak{g}$ on $\mathfrak{g}^*$ given by $(\ad^*(A) \beta)(B)=- \beta([A,B])$ for any $A, B \in \mathfrak{g}$, $\beta \in \mathfrak{g}^*$.

\begin{remark}\label{re:normal-abnormal}
We remark the following about the result of Theorem~\ref{th:harmonic}. 
\begin{enumerate}[$\bullet$]
\item Case (a), which we call \emph{abnormal}, occurs when $\alpha = \alpha_f \in \calA_{D,E}$ is singular. It is a property that holds for all singular elements, and it is not related to optimality. The proof of the result in (a) does not use the property that $\alpha$ is a harmonic form. 
\item Case (b), which is called \emph{normal}, occurs when $\alpha = \alpha_f \in \calA_{D,E}$ is both regular and is a critical value of $\hat \scrE$. However, there are also singular forms $\alpha \in \calA_{D,E}$ that are critical values of $\hat \scrE$ and also have a corresponding $\lambda \in \Omega^1(M, \mathfrak{g}^*)$ satisfying the equations as in (b), but such an extremal form $\alpha$ is not called normal. Thus, Cases (a) and (b) are not completely disjoint.
\item We remark also that the results of Theorem~\ref{th:harmonic} only depend on restrictions $\alpha|D$ and $\lambda|D$ of forms to $D$ and do not depend on their extension to the entire tangent bundle. Hence, we could have stated Theorem~\ref{th:harmonic} by using only the restrictions $\alpha|D$ and $\lambda|D$.
\end{enumerate}
Before proceeding to the proof, we observe how the result of Theorem~\ref{th:harmonic} satisfies known examples in literature.
\end{remark}

\begin{example}
We note that if $E = TG$, so that $G$ is a Riemannian Lie group, then there cannot exist any abnormal harmonic maps. Indeed, since $\sharp$ is now is a bijective map, we cannot have that  $\eta|D \neq 0$, while still having that $\sharp \eta|D =0$.

For the normal case, we have $\sharp \lambda = \alpha$ and if $\alpha =\alpha_f$, then
\begin{equation} \label{eqLambdaRiemannian} \delta_D \alpha + \tr_g \ad(\alpha(\times))^{\dagger}\alpha(\times) =0,\end{equation}
where $\ad(\alpha(\times))^{\dagger}$ is the transpose map with respect to left-invariant metric~$h$ on~$G$.
This is just the classical tension field equation for harmonic maps. In order to explain this, we write \eqref{eqLambdaRiemannian} as
\begin{equation}\label{eq:tau}
\tau(f) = \tr_{g} \bnabla_{\times} df(\times) = 0, \qquad \bnabla = \nabla \otimes f^* \nabla^h.
\end{equation}
Here, $\nabla^h$ is the Levi-Civita connection on $G$, which for left-invariant vector fields, can be written as,
$$2\nabla_A^h B = \ad(A)B  - \ad(A)^{\dagger} B - \ad(B)^{\dagger} A.$$
Furthermore, $\bnabla$ is the induced connection on $T^*M \otimes f^* TG$, which can be described by
$$(\bnabla_{X} df)(Y) = \nabla_{df(X)}^h df(Y) - df(\nabla_X Y).$$
Equation~\eqref{eq:tau} coincides with the tension field $\tau(f)$ for maps between the Riemannian manifolds in~\cite{MR164306} or from sub-Riemannian manifolds to Riemannian manifolds in~\cite{MR4236537,MR1433120}. For the special case
$G = \mathbb{R}$, we have that $\tau(f) = \Delta_{g,d\mu} f$.
\end{example}
\begin{example} Consider $M = [0,1]$. Although this is not within the scope of the theorem, as~$M$ is a Riemannian manifold with boundary, the theorem is still valid under the assumption that any variation is constant on $\partial M = \{0, 1\}$. Define $\nabla^\ell$ to be the left-invariant connection on $G$, i.e., the connection such that $\nabla^\ell A = 0$ for any left-invariant vector field $A$. This connection is compatible with the sub-Riemannian structure $(E,h)$. Let $T^\ell$ be the torsion of $\nabla^\ell$, given for left-invariant vector fields by
$$T^\ell(A,B) = -\ad(A) B.$$
We say that the \emph{adjoint connection} to $\nabla^\ell$ is given by $\hat \nabla_A^\ell B = \nabla_A^\ell B - T^\ell(A,B)$. For the special case of $\nabla^\ell$, its adjoint will be the right invariant connection. If $f\colon [0,1] \to G$, then the equation in Theorem~\ref{th:harmonic}~(a) is written as
$$\hat \nabla_{\dot f} \eta =0,\qquad \sharp^h \eta = 0,$$
where $\eta$ is a one-form along $f(t)$. 
The equation in Theorem~\ref{th:harmonic}~(b) becomes
$$\hat \nabla_{\dot f} \lambda =0,\qquad \sharp^h \lambda = \dot f.$$
These are the respective equations for abnormal curves and normal geodesics see \cite{GoGr15,grong2020affine} for details.
\end{example}

\begin{proof}[Proof of Theorem~\ref{th:harmonic}]
Recall that we have defined $L^2$-forms with respect to a taming metric $\bar{g}$ and an inner product on $\mathfrak{g}$.
Assume that $\alpha \in \calA_{D,E}$ is harmonic. Write
\begin{align*}
Q & = \{ \eta \in \Omega^1(M,\mathfrak{g}^*) \, : \eta(D) \subseteq \Ann(\mathfrak{e}) \}, \\
\Lambda_{\alpha} &= \{ \lambda \in \Omega^1(M, \mathfrak{g}^*) \, : \, \sharp \lambda|D=\alpha|D \}.
\end{align*}
Note that $Q$ is a vector space, while $\Lambda_{\alpha}$ is an affine space with $\lambda_1 - \lambda_2 \in Q$ for $\lambda_1, \lambda_2 \in \Lambda_\alpha$.
We first observe the following. Consider the operator $L_\alpha F := dF + [\alpha, F]$. Then for any $\eta \in L^2 \Omega^1(M, \mathfrak{g}^*)$, we note that
\begin{align*}
\int_M \tr_g \eta(\times) dF( \times) d\mu &= \int_M \tr_g \eta(\times) \nabla_{\times} F d\mu = \int_M (\delta_D\eta) F d\mu;
\end{align*}
hence
\begin{align*}
\int_M \tr_g \eta(\times) (L_\alpha F)(\times)\,d\mu & = \int_M \tr_g \eta(\times)( dF(\times) + [\alpha(\times), F] ) \\
& = \int_M (\delta_D \eta - \tr_{g} \ad^*(\alpha(\times) )\eta(\times)) F d\mu  =: \int_M (L_\alpha^* \eta) F d\mu.
\end{align*}

If $\alpha \in \overline{\calA_{D,E}}$ is singular, then there is a non-zero form $\check{\eta} \in L^2 \Omega^1(D, \mathfrak{g}/\mathfrak{e})$ orthogonal to the image of $D_\alpha P(T_\alpha \calA)$ where $P$ is given in~\eqref{eq:P}. Define $\eta = \langle \check{\eta} , D_\alpha P \cdot \rangle_{L^2} \in \bar{Q}$. Then for any element $L_\alpha F$ in $T_\alpha \calA$, $F \in C^\infty(M,\mathfrak{g})$, we have
$$0 =\langle \eta, D_\alpha P L_\alpha F \rangle = \int_M \eta(L_\alpha F)\, d\mu = \int_M (L_\alpha^* \eta) F\, d\mu.
$$
As this holds for any $F \in C^\infty(M, \mathfrak{g})$, the result in (a) follows.

If $\alpha$ is regular, then $\overline{\calA_{D,E}}$ is locally a manifold with $T_\alpha \overline{\calA_{D,E}}$ being the closure of elements $L_\alpha F$ such that $L_\alpha F(D) \subseteq \mathfrak{e}$. In other words, elements in $T_\alpha \overline{\calA_{D,E}}$ are in the closure of elements $L_\alpha F$, $F \in C^\infty(M, \mathfrak{g})$, that are orthogonal to $Q$, which can be written as
$$T_{\alpha} \overline{\calA_{D,E}} = \overline{\left\{ L_\alpha F \, :\, \langle F, \phi \rangle_{L^2} =0 \text{ for any } \phi \in L_\alpha^* Q \right\}}.$$
Let $F$ be an arbitrary such element in $C^\infty(M,\mathfrak{g})$ that is orthogonal to $L_\alpha^* Q$. For such a tangent vector in $T_{\alpha} \overline{\calA_{D,E}}$, let $\alpha_s$ be the corresponding variation with $\alpha_0 = \alpha$ and $\frac{d}{ds} \alpha_s |_{s=0} = L_\alpha F$.
We observe that for any smooth $\tilde \lambda \in \Lambda_\alpha$,
\begin{align*}
\frac{d}{ds} \hat \scrE(\alpha_s) |_{s=0} &
= \int_M \tr_g \langle \alpha(\times) , dF(\times) + [\alpha(\times), F] \rangle d\mu = \int_M (L_\alpha^* \tilde \lambda) F d\mu. \end{align*} 
If this vanishes for all such variations, then $L_\alpha^* \tilde \lambda \in ( C^\infty(M,\mathfrak{g}) \cap (L_\alpha^* Q)^\perp)^\perp$. We remark that since the elements of $L^2(M,\mathfrak{g})$ can be considered as equivalence classes of sequences $(\phi_n)_{n=1}^\infty$ of smooth functions convergent in $L^2$, we have
$$( C^\infty(M,\mathfrak{g}) \cap (L_\alpha^* Q)^\perp)^\perp = \overline{L_\alpha^* Q}, \qquad \text{ and } \qquad  \overline{L_\alpha^* Q} \cap C^\infty(M,\mathfrak{g}) = L_\alpha^* Q.$$
Furthermore, since $\tilde \lambda$ is smooth, then so is $L_\alpha^* \tilde  \lambda$; hence we can write $L_\alpha^* \tilde \lambda = L_\alpha^* \eta \in L_\alpha^* Q$. By defining $\lambda = \tilde \lambda - \eta \in \Lambda_\alpha$, we find that $\lambda$ satisfies $L_\alpha^* \lambda = 0$.
\end{proof}

\begin{remark}\label{re:NSC}
The results in Theorem~\ref{th:harmonic} can be generalized to a non simply connected manifold. If $M$ is not simply connected, we consider its universal cover $\Pi\colon \tilde M \to M$. We note that $\tilde M$ might not be compact, but, as mentioned in our introduction, we can consider compact subdomains. We can then lift functions from $M$ to $\tilde M$ as $f \mapsto f \circ \Pi$. By using a partition of unity, we decompose an integral over $\tilde M$ or one of its compact subdomains as integrals over open sets that are mapped bijectively to an open set in $M$. It leads to the conclusion that if $f$ is a harmonic map, then so is $f \circ \Pi$. Looking at the equations in Theorem~\ref{th:harmonic}, we see that they are all local and can hence be projected to $M$.
\end{remark}

\subsection{Harmonic maps into the Heisenberg group}
We consider the case of harmonic maps $f\colon M \to H^{n}$. Let $(a,b,c)$ be the coordinates on $H^n$ as described in Example~\ref{ex:Hn}. We then have the following corollary.
\begin{proposition}\label{prop:Heisenberg}
Let $f\colon M \to H^n$ be a horizontal map from $(M,D,g,d\mu)$ into the Heisenberg group $(H^n, E,h)$ with its standard sub-Riemannian structure. Write
$$(u,v,w) = (a,b,c) \circ f, \qquad \zeta= u +iv.$$
Then $f$ is a normal harmonic if and only if for some horizontal vector field $Y \in \Gamma(M)$ satisfying $\dv_{d\mu} Y =0$ we have
$$(\Delta_{g,d\mu} - i Y)\zeta = 0.$$
\end{proposition}
We note that the operator $\Delta_{g,d\mu} - iY$ is hypoelliptic by~\cite{BN85}. However, recall that we are \emph{also} assuming that $f$ is horizontal, meaning that 
$$\textstyle \left(dw + \frac{1}{2} \sum_{j=1}^n (v_j du_j - u_j dv_j)\right) |D=0.$$
Note that if $\rank D \leq 2n$, then all harmonic maps are normal by Example~\ref{ex:sb}. and Proposition~\ref{prop:sb}.
\begin{proof}
Since $f$ is horizontal, then
$$\alpha|D = \alpha_f|D = \sum_{j=1}^n du_j|D \otimes A_j + \sum_{j=1}^n dv_j|D \otimes B_j.$$
From the requirement that $\sharp \lambda|_D = \alpha|D$, we have that
$$\lambda|D = \sum_{j=1}^n du_j|D \otimes da_j + \sum_{j=1}^n dv_j|D \otimes db_j + \lambda_0|D \otimes \theta,$$
where $\lambda_0$ is a one-form on $M$. Write $Y = \sharp^g \lambda_0$ as a vector field. The harmonic equation is given by
\begin{align*}
& 0 =\delta_D \lambda - \tr_{g} \ad^*(\alpha(\times) )\lambda(\times) \\
& = - \sum_{j=1}^n \tr_g (\nabla_{\times} du_j)(\times) \otimes da_j - \sum_{j=1}^n \tr_g (\nabla_{\times} dv_j)(\times) \otimes db_j + (\nabla_{\times} \lambda_0)(\times) \otimes \theta \\
& \qquad + \tr_g \lambda_0(\times)  \left(du_j(\times) \otimes db_j  \right)  - \tr_g \lambda_0(\times)  \left(dv_j(\times) \otimes da_j  \right) \\
& = - \sum_{j=1}^n (\Delta_D u_j + Yv_j) \otimes da_j - \sum_{j=1}^n (\Delta_D v_j - Y u_j) \otimes db_j + (\dv_\mu Y) \otimes \theta.
\end{align*}
It follows that $\dv_\mu Y = 0$ and 
$$\Delta_{g,\mu} \zeta = (\Delta_{g,\mu} u_j + i \Delta_{g,\mu} v_j )_j = (- Yv_j + i Yu_j)_j = i Y\zeta,
$$
completing the proof.
\end{proof}

\begin{example}
If we choose $M = [0,1]$ in Proposition~\ref{prop:Heisenberg}, allowing a manifold with boundary, we obtain the normal sub-Riemannian geodesics on the Heisenberg group. More precisely, the horizontality requirement for $f$ can be written as
$$\dot w = - \frac{1}{2} \sum_{j=1}^n (v_j \dot u_j - u_j \dot v_j),$$
while $\zeta = u + iv$ now has to satisfy
$$\ddot \zeta + i y\dot \zeta =0, \qquad \text{for some constant $y$}.$$
In other words $\dot \zeta = e^{iyt} \dot \zeta(0)$, which is exactly the equation for normal geodesics on the Heisenberg group with solutions being the horizontal lifts of circular arcs on the $(u,v)$-plain.
\end{example}

\appendix

\section{Forms with values in the Lie algebra} \label{sec:Forms}
We recall here the definition of brackets of forms with values in a Lie algebra $\mathfrak{g}$. Let $\alpha \in \Omega^k(M, \mathfrak{g})$ be a
$\mathfrak{g}$ valued $k$-form, that is a section of the vector bundle $\wedge^k T^*M \otimes \mathfrak{g}$.  Note that all such elements can be written as a finite sum of elements $\check{\alpha} \otimes A$ where $\check{\alpha} \in \Omega^k(M)$ is a real valued $k$-form, and $A \in \mathfrak{g}$. For $\check \alpha \in \Omega^k(M)$, $\check \beta \in \Omega^l(M)$ and $A,B \in \mathfrak{g}$, we define
$$\left[ \check{\alpha} \otimes A, \check{\beta} \otimes B\right] =(\check{\alpha} \wedge \check{\beta}) \otimes [A,B].$$
We can extend this definition by linearity to arbitrary forms $\alpha \in \Omega^k(M, \mathfrak{g})$ and $\beta \in \Omega^l(M, \mathfrak{g})$, to obtain a form $[\alpha,\beta] \in \Omega^{k+l}(M, \mathfrak{g})$. We note that $[\alpha, \beta] = (-1)^{kl+1} [\beta, \alpha]$. We look at the particular case when $k =l=1$. For $\alpha, \beta \in \Omega^1(M,\mathfrak{g})$ and a basis $A_1, \dots, A_n$ of $\mathfrak{g}$, we write $\alpha = \sum_{j=1}^n \check \alpha_j \otimes A_j$ and $\beta = \sum_{j=1}^n \check \beta_j \otimes A_j$. We then observe that for any $v , w \in TM$,
\begin{align*}
& {[\alpha,\beta]}(v,w) = [\beta,\alpha](v,w)  = \sum_{i,j=1}^n (\check \alpha_i(v) \check \beta_j(w) - \check \beta_j(v) \check \alpha_i(w)) [A_i, A_j] \\
& = [\alpha(v), \beta(w)] - [\alpha(w), \beta(v)] = [\alpha(v), \beta(w)] + [\beta(v), \alpha(w)].
\end{align*}
In particular, $[\alpha,\alpha](v,w) =2 [\alpha(v),\alpha(w)]$. If $\alpha$ is a one-form and $F$ is a zero-form, i.e., a function, then $[\alpha,F](v) = [\alpha(v), F(x)]$ for every $v \in T_xM$.

\

\paragraph{\bf Acknowledgements} We thank Pierre Pansu and Mauricio Godoy Molina for helpful discussions.

\bibliographystyle{habbrv}
\bibliography{Bibliography,BibliographyHarmonic}

\end{document}